\documentclass[11pt]{article}

\usepackage{amsmath, amssymb, amsthm, amsfonts}
\usepackage[T1]{fontenc}
\usepackage{mathpazo}         
\usepackage{microtype}        
\usepackage[outline]{contour} 

\usepackage{graphicx}
\usepackage[dvipsnames]{xcolor}
\usepackage{enumitem}
\usepackage{lastpage}
\usepackage{fancyhdr}
\usepackage[unicode, hidelinks]{hyperref}
\usepackage[framemethod=TikZ]{mdframed}
\usetikzlibrary{decorations.pathreplacing}
\usepackage{fontawesome5}
\usepackage{float}

\newcommand{\makeboxedtitle}[4]{
    \begin{center}
        \begin{mdframed}[
            linecolor=black,
            linewidth=1pt,
            roundcorner=0pt,
            innertopmargin=15pt,
            innerbottommargin=15pt,
            backgroundcolor=gray!5
        ]
            \centering
            {\LARGE \bfseries #1 \par} 
            \vspace{12pt}
            {\large 
              \begin{tabular}{c} 
                #2 
              \end{tabular}\par} 
            \ifx&#3&\else \vspace{8pt} {\small \texttt{#3} \par} \fi
            \vspace{5pt}
            {\small \today}
        \end{mdframed}
        \vspace{10pt}
        \begin{abstract}
            #4
        \end{abstract}
    \end{center}
    \vspace{15pt}
}

\fancypagestyle{firstpage}{
    \fancyhf{}
    
    \fancyfoot[L]{\tiny Typeset with \href{https://github.com/tonamatos/tonas-latex}{tonas-papers.sty}}
    \fancyfoot[C]{\footnotesize \copyright\ \the\year\ Tonatiuh Matos-Wiederhold}
}

\AtBeginDocument{\thispagestyle{firstpage}}

\theoremstyle{plain}
\newtheorem{theorem}{Theorem}[section]
\newtheorem{lemma}[theorem]{Lemma}

\newtheorem{corollary}[theorem]{Corollary}
\newtheorem{fact}[theorem]{Fact}

\newtheorem{claim}{Claim}[theorem] 

\theoremstyle{definition}
\newtheorem{definition}[theorem]{Definition}

\theoremstyle{remark}

\newenvironment{claimproof}[1][Proof of Claim]{
    \begin{proof}[#1]
}{
    \end{proof}
}

\DeclareMathOperator{\Hom}{Hom}

\begin{document}

\makeboxedtitle{A Concise Proof of the $L_0$ Dichotomy}{
    Tonatiuh Matos-Wiederhold
}{}{
    Carroy, Miller, Schrittesser, and Vidny\'anszky established the $L_0$ dichotomy: there is a Borel graph of Borel chromatic number three that admits a continuous homomorphism to every analytic graph of Borel chromatic number at least three. Their proof relies on a transfinite analysis of terminal approximations over a decreasing $\omega_1$-sequence of analytic sets.

    I give a new, substantially shorter proof of this result by adapting the graph-theoretic framework recently introduced by Bernshteyn for the $G_0$ dichotomy. The central device is a $\sigma$-ideal of \emph{small} sets of homomorphisms from finite path approximations into the target graph, where smallness is witnessed by a bounded odd-walk condition on vertex projections. The key lemma that largeness is preserved under the doubling operation is established via the First Reflection Theorem, replacing the original transfinite construction with a single Borel reflection argument. The continuous homomorphism from the canonical graph $\mathbb L_c$ into the target is then obtained as a limit of shrinking families of copies, in direct analogy with Bernshteyn's proof for $G_0$.
}

\section{Introduction}\label{sec:concise-intro}

Throughout this entire work, $G$ denotes an analytic graph on a Polish space $V(G)=X$.
That is, $X$ is a separable completely metrizable topological space, and $G$ is a symmetric and irreflexive relation on $X$ which, as a subspace of $X\times X$, is the continuous image of a Polish space.

I denote the \emph{Borel chromatic number} of $G$ by $\chi_B(G)$, that is, $\chi_B(G)$ is the least number of colours $k$ needed to find a Borel mapping $c\colon X\to k$ in such a way that adjacent vertices in $G$ are given different colours.
I employ the symbol $H\to_cG$ (respectively, $H\to_BG$) to abbreviate the fact that there is a continuous (respectively Borel) homomorphism from the analytic graph $H$ to the analytic graph $G$.
It is straightforward to show the following.

\begin{fact}\label{fact:chrom_transitive}
    If $H\to_cG$ or, more generally if $H\to_BG$, then $\chi_B(H)\leq\chi_B(G)$.
\end{fact}

Using a standard greedy algorithm argument, it is clear that any finite graph $G$ of maximum degree $\Delta$ satisfies $\chi(G)\leq\Delta+1$.
Interestingly, a similar fact holds for Borel graphs of uniformly bounded degree.

\begin{theorem}[Proposition~4.6 in \cite{kechris1999borel}]\label{prop:brooks}
    If $G$ is a Borel graph on a Polish space all of whose degrees are bounded by the natural number $k$, then $\chi_B(G)\leq k+1$.
\end{theorem}

The following notion, taken from \cite[Theorem~35.10, p.~285]{Kechris_1995_DescriptiveSetTheory}, is a well-known fact about analytic sets.

\begin{definition}\label{def:pi11-on-sigma11}
    A collection of subsets $\Phi$ of a Polish space $X$ is said to be \emph{$\Pi^1_1$ on $\Sigma^1_1$} if for any Polish $Y$ and any $\Sigma^1_1$ set $A\subseteq Y\times X$, the set $A_\Phi:=\{y\in Y:A_y:=\{x\in X:(y,x)\in A\}\in\Phi\}$ is $\Pi^1_1$.
\end{definition}

\begin{lemma}[First Reflection Theorem]\label{lemma:first-reflection}
    If $\Phi$ is as in the preceding definition, then for any $\Sigma^1_1$ set $A\in\Phi$, there is a Borel set $B\subseteq X$ such that $A\subseteq B$ and $B\in\Phi$.
\end{lemma}

\section{Finite path approximations}\label{sec:finite-path-approx}

Fix the parameter $c\in\omega^\omega$ and a path of length $\omega$ on the vertex set $\{p_n\}_{n<\omega}$.
That is, the edge set $\{(p_n,p_{n+1}):n<\omega\}$.
Naturally, one could assume that $p_n=n$, but the labeling helps distinguish these vertices from other indices.

Recall that if $s$ and $t$ are finite strings (of, say, natural numbers), then $s^\frown t$ is the \emph{concatenation} of $s$ and $t$.
Similarly, $t^n$ is the repeated concatenation of $t$ with itself, and $(i)$ denotes the string of single value $i$.
I recursively construct a family of graphs $\{L^c_n:n<\omega\}$ such that for all $n<\omega$, the following properties hold.

\begin{enumerate}[label=(\roman*)]
    \item $L^c_n$ is a finite path.
    \item If $n>0$, the endpoints of $L^c_n$ are $e_i^n:=(p_0)^\smallfrown(0)^{n-1}{}^\smallfrown(i)$, for $i<2$.
    \item Every vertex of $L^c_n$ is of the form $(p_k)^\smallfrown t$ for some $k$ and $t\in2^{\leq n}$. Unless $t=\emptyset$, I call the vertex a \emph{non-path} vertex.
\end{enumerate}

Start with $L^c_0$ as the graph consisting of a single vertex $p_0$ and no edges, and set $e^0_0=e^0_1=p_0$.
Now suppose that $L^c_n$ has been constructed.
$L^c_{n+1}$ will be a path obtained from joining two copies of $L^c_n$ by a path of $c(n)+2$ edges (with $c(n)+1$ \emph{new} path vertices), but I must update the vertex labels in order to distinguish the two copies.
I do this by simply appending $0$ to each vertex of the first copy, and $1$ to the other.
Formally, I start by adding the vertex $v^\frown(i)$ for every $i<2$ and $v\in L^c_n$.
Adding the path $$((e_1^n)^\smallfrown(0),p_0,p_1,\dots,p_{c(n)},(e_1^n)^\smallfrown(1))$$ completes the construction of $L^c_{n+1}$, and its endpoints are $$e_i^{n+1}=(e_0^n)^\smallfrown(i)=(p_0)^\smallfrown(0)^n{}^\smallfrown(i),$$ for $i<2$.
This finishes the recursion.
Note that to construct $L^c_{n+1}$ from $L^c_n$, I only require knowledge of the value $c(n)$.

\begin{figure}[ht]
    \centering
    \begin{tikzpicture}[
    scale=0.78, transform shape,
    copycolor/.style={fill=blue!18, draw=black!70},
    joinone/.style={fill=green!28, draw=black!70},
    jointwo/.style={fill=orange!35, draw=black!70},
    jointhree/.style={fill=red!30, draw=black!70},
    nd/.style={circle, minimum size=5.5pt, inner sep=0pt, font=\scriptsize},
    every edge/.style={draw=black!50, thick},
]

    \draw[thick, black!40] (0.00,0.00) -- (1.60,0.00);
    \draw[thick, black!40] (1.60,0.00) -- (3.20,0.00);
    \draw[thick, black!40] (3.20,0.00) -- (4.80,0.00);
    \draw[thick, black!40] (4.80,0.00) -- (6.40,0.00);
    \draw[thick, black!40] (6.40,0.00) -- (8.00,0.00);
    \draw[thick, black!40] (8.00,0.00) -- (8.96,-0.65);
    \draw[thick, black!40] (8.96,-0.65) -- (8.96,-1.30);
    \draw[thick, black!40] (8.96,-1.30) -- (8.96,-1.95);
    \draw[thick, black!40] (8.96,-1.95) -- (8.96,-2.60);
    \draw[thick, black!40] (8.96,-2.60) -- (8.00,-3.30);
    \draw[thick, black!40] (8.00,-3.30) -- (6.40,-3.30);
    \draw[thick, black!40] (6.40,-3.30) -- (4.80,-3.30);
    \draw[thick, black!40] (4.80,-3.30) -- (3.20,-3.30);
    \draw[thick, black!40] (3.20,-3.30) -- (1.60,-3.30);
    \draw[thick, black!40] (1.60,-3.30) -- (0.00,-3.30);
    \draw[thick, black!40] (0.00,-3.30) -- (-0.96,-3.95);
    \draw[thick, black!40] (-0.96,-3.95) -- (-0.96,-4.60);
    \draw[thick, black!40] (-0.96,-4.60) -- (-0.96,-5.25);
    \draw[thick, black!40] (-0.96,-5.25) -- (-0.96,-5.90);
    \draw[thick, black!40] (-0.96,-5.90) -- (-0.96,-6.55);
    \draw[thick, black!40] (-0.96,-6.55) -- (-0.96,-7.20);
    \draw[thick, black!40] (-0.96,-7.20) -- (0.00,-7.90);
    \draw[thick, black!40] (0.00,-7.90) -- (1.60,-7.90);
    \draw[thick, black!40] (1.60,-7.90) -- (3.20,-7.90);
    \draw[thick, black!40] (3.20,-7.90) -- (4.80,-7.90);
    \draw[thick, black!40] (4.80,-7.90) -- (6.40,-7.90);
    \draw[thick, black!40] (6.40,-7.90) -- (8.00,-7.90);
    \draw[thick, black!40] (8.00,-7.90) -- (8.96,-8.55);
    \draw[thick, black!40] (8.96,-8.55) -- (8.96,-9.20);
    \draw[thick, black!40] (8.96,-9.20) -- (8.96,-9.85);
    \draw[thick, black!40] (8.96,-9.85) -- (8.96,-10.50);
    \draw[thick, black!40] (8.96,-10.50) -- (8.00,-11.20);
    \draw[thick, black!40] (8.00,-11.20) -- (6.40,-11.20);
    \draw[thick, black!40] (6.40,-11.20) -- (4.80,-11.20);
    \draw[thick, black!40] (4.80,-11.20) -- (3.20,-11.20);
    \draw[thick, black!40] (3.20,-11.20) -- (1.60,-11.20);
    \draw[thick, black!40] (1.60,-11.20) -- (0.00,-11.20);

    \node[nd, copycolor] (v0) at (0.00,0.00) {};
    \node[above=1pt, font=\tiny] at (v0) {$(0{,}0{,}0{,}0)$};
    \node[nd, copycolor] (v1) at (1.60,0.00) {};
    \node[above=1pt, font=\tiny] at (v1) {$(1{,}0{,}0{,}0)$};
    \node[nd, joinone] (v2) at (3.20,0.00) {};
    \node[above=1pt, font=\tiny] at (v2) {$(0{,}0{,}0)$};
    \node[nd, joinone] (v3) at (4.80,0.00) {};
    \node[above=1pt, font=\tiny] at (v3) {$(1{,}0{,}0)$};
    \node[nd, copycolor] (v4) at (6.40,0.00) {};
    \node[above=1pt, font=\tiny] at (v4) {$(1{,}1{,}0{,}0)$};
    \node[nd, copycolor] (v5) at (8.00,0.00) {};
    \node[above=1pt, font=\tiny] at (v5) {$(0{,}1{,}0{,}0)$};
    \node[nd, jointwo] (v6) at (8.96,-0.65) {};
    \node[right=1pt, font=\tiny] at (v6) {$(0{,}0)$};
    \node[nd, jointwo] (v7) at (8.96,-1.30) {};
    \node[right=1pt, font=\tiny] at (v7) {$(1{,}0)$};
    \node[nd, jointwo] (v8) at (8.96,-1.95) {};
    \node[right=1pt, font=\tiny] at (v8) {$(2{,}0)$};
    \node[nd, jointwo] (v9) at (8.96,-2.60) {};
    \node[right=1pt, font=\tiny] at (v9) {$(3{,}0)$};
    \node[nd, copycolor] (v10) at (8.00,-3.30) {};
    \node[below=1pt, font=\tiny] at (v10) {$(0{,}1{,}1{,}0)$};
    \node[nd, copycolor] (v11) at (6.40,-3.30) {};
    \node[below=1pt, font=\tiny] at (v11) {$(1{,}1{,}1{,}0)$};
    \node[nd, joinone] (v12) at (4.80,-3.30) {};
    \node[below=1pt, font=\tiny] at (v12) {$(1{,}1{,}0)$};
    \node[nd, joinone] (v13) at (3.20,-3.30) {};
    \node[below=1pt, font=\tiny] at (v13) {$(0{,}1{,}0)$};
    \node[nd, copycolor] (v14) at (1.60,-3.30) {};
    \node[below=1pt, font=\tiny] at (v14) {$(1{,}0{,}1{,}0)$};
    \node[nd, copycolor] (v15) at (0.00,-3.30) {};
    \node[below=1pt, font=\tiny] at (v15) {$(0{,}0{,}1{,}0)$};
    \node[nd, jointhree] (v16) at (-0.96,-3.95) {};
    \node[left=1pt, font=\tiny] at (v16) {$(0)$};
    \node[nd, jointhree] (v17) at (-0.96,-4.60) {};
    \node[left=1pt, font=\tiny] at (v17) {$(1)$};
    \node[nd, jointhree] (v18) at (-0.96,-5.25) {};
    \node[left=1pt, font=\tiny] at (v18) {$(2)$};
    \node[nd, jointhree] (v19) at (-0.96,-5.90) {};
    \node[left=1pt, font=\tiny] at (v19) {$(3)$};
    \node[nd, jointhree] (v20) at (-0.96,-6.55) {};
    \node[left=1pt, font=\tiny] at (v20) {$(4)$};
    \node[nd, jointhree] (v21) at (-0.96,-7.20) {};
    \node[left=1pt, font=\tiny] at (v21) {$(5)$};
    \node[nd, copycolor] (v22) at (0.00,-7.90) {};
    \node[above=1pt, font=\tiny] at (v22) {$(0{,}0{,}1{,}1)$};
    \node[nd, copycolor] (v23) at (1.60,-7.90) {};
    \node[above=1pt, font=\tiny] at (v23) {$(1{,}0{,}1{,}1)$};
    \node[nd, joinone] (v24) at (3.20,-7.90) {};
    \node[above=1pt, font=\tiny] at (v24) {$(0{,}1{,}1)$};
    \node[nd, joinone] (v25) at (4.80,-7.90) {};
    \node[above=1pt, font=\tiny] at (v25) {$(1{,}1{,}1)$};
    \node[nd, copycolor] (v26) at (6.40,-7.90) {};
    \node[above=1pt, font=\tiny] at (v26) {$(1{,}1{,}1{,}1)$};
    \node[nd, copycolor] (v27) at (8.00,-7.90) {};
    \node[above=1pt, font=\tiny] at (v27) {$(0{,}1{,}1{,}1)$};
    \node[nd, jointwo] (v28) at (8.96,-8.55) {};
    \node[right=1pt, font=\tiny] at (v28) {$(0{,}1)$};
    \node[nd, jointwo] (v29) at (8.96,-9.20) {};
    \node[right=1pt, font=\tiny] at (v29) {$(1{,}1)$};
    \node[nd, jointwo] (v30) at (8.96,-9.85) {};
    \node[right=1pt, font=\tiny] at (v30) {$(2{,}1)$};
    \node[nd, jointwo] (v31) at (8.96,-10.50) {};
    \node[right=1pt, font=\tiny] at (v31) {$(3{,}1)$};
    \node[nd, copycolor] (v32) at (8.00,-11.20) {};
    \node[below=1pt, font=\tiny] at (v32) {$(0{,}1{,}0{,}1)$};
    \node[nd, copycolor] (v33) at (6.40,-11.20) {};
    \node[below=1pt, font=\tiny] at (v33) {$(1{,}1{,}0{,}1)$};
    \node[nd, joinone] (v34) at (4.80,-11.20) {};
    \node[below=1pt, font=\tiny] at (v34) {$(1{,}0{,}1)$};
    \node[nd, joinone] (v35) at (3.20,-11.20) {};
    \node[below=1pt, font=\tiny] at (v35) {$(0{,}0{,}1)$};
    \node[nd, copycolor] (v36) at (1.60,-11.20) {};
    \node[below=1pt, font=\tiny] at (v36) {$(1{,}0{,}0{,}1)$};
    \node[nd, copycolor] (v37) at (0.00,-11.20) {};
    \node[below=1pt, font=\tiny] at (v37) {$(0{,}0{,}0{,}1)$};

    \draw[->, thick, blue!60] (0.00,0.55) -- (1.60,0.55);
    \draw[->, thick, blue!60] (8.00,0.55) -- (6.40,0.55);
    \draw[->, thick, blue!60] (8.00,-3.85) -- (6.40,-3.85);
    \draw[->, thick, blue!60] (0.00,-3.85) -- (1.60,-3.85);
    \draw[->, thick, blue!60] (0.00,-7.35) -- (1.60,-7.35);
    \draw[->, thick, blue!60] (8.00,-7.35) -- (6.40,-7.35);
    \draw[->, thick, blue!60] (8.00,-11.75) -- (6.40,-11.75);
    \draw[->, thick, blue!60] (0.00,-11.75) -- (1.60,-11.75);

    \draw[decorate, decoration={brace, amplitude=6pt, mirror}, thick, black!50]
        (9.76,0.40) -- (9.76,-3.70)
        node[midway, right=8pt, font=\footnotesize] {$L^c_2$ copy $0$};
    \draw[decorate, decoration={brace, amplitude=6pt, mirror}, thick, black!50]
        (9.76,-7.50) -- (9.76,-11.60)
        node[midway, right=8pt, font=\footnotesize] {$L^c_2$ copy $1$};
    \node[font=\footnotesize, text=red!60!black] at (-1.76,-5.57) {\rotatebox{90}{join $c(2){=}5$}};
    \node[font=\tiny, text=orange!70!black] at (9.66,-1.62) {\rotatebox{90}{join $c(1){=}3$}};
    \node[font=\tiny, text=orange!70!black] at (9.66,-9.52) {\rotatebox{90}{join $c(1){=}3$}};

    \node[nd, copycolor, minimum size=7pt] at (0,-12.40) {};
    \node[right=3pt, font=\scriptsize] at (0.15,-12.40) {$L_0$ copy};
    \node[nd, joinone, minimum size=7pt] at (2.0,-12.40) {};
    \node[right=3pt, font=\scriptsize] at (2.15,-12.40) {level-1 join};
    \node[nd, jointwo, minimum size=7pt] at (4.4,-12.40) {};
    \node[right=3pt, font=\scriptsize] at (4.55,-12.40) {level-2 join};
    \node[nd, jointhree, minimum size=7pt] at (6.8,-12.40) {};
    \node[right=3pt, font=\scriptsize] at (6.95,-12.40) {level-3 join};
\end{tikzpicture}
    \caption{$L_3^c$ for $c(0)=1,c(1)=3,c(2)=5$.}
    \label{fig:L_n}
\end{figure}
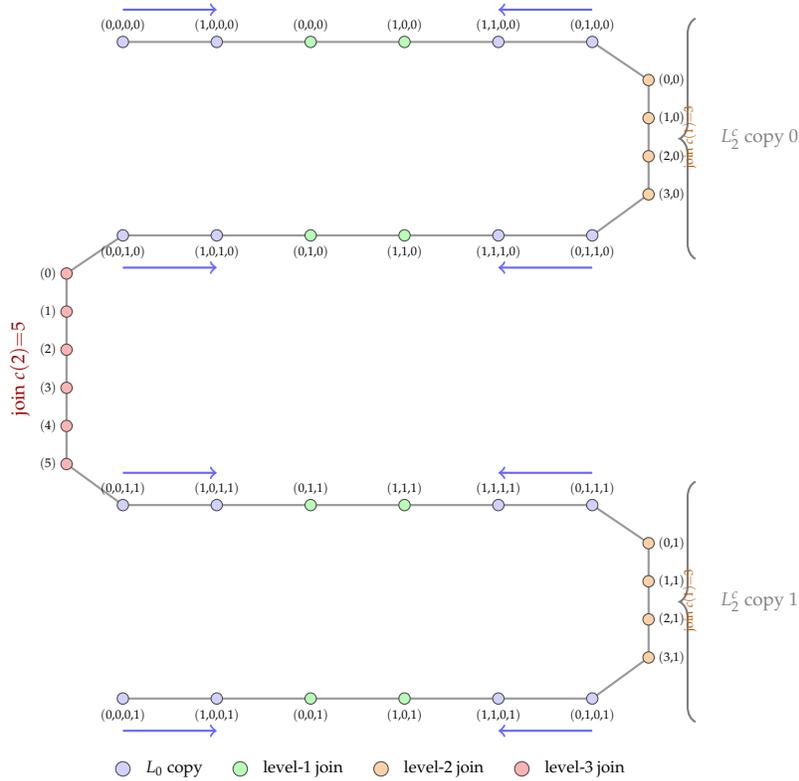

To develop some intuition on what (iii) says about a vertex $v$ in $L^c_n$: the number $|v|-1$ indicates that $v$ was added that many stages ago in the recursion, at the position of $p_k$ in the path joining the two copies of the previous paths; the sequence $t$ tells the story, in order, of which of the two copies of the previous path the vertex lies in.
In other words, vertices of the form $(p_k)^\smallfrown t$ for $t\in 2^n$ are precisely those that were copied from the original $L^c_0$ in all possible positions along $L^c_n$, doubling the number of copies at each stage.

Based on this discussion, a vertex in some $L^c_n$ is completely determined by the stage $n-m$ at which it was added in the position of $p_k$ and then copied according to the binary sequence $t$.

\section{A notion of smallness for copies}\label{sec:notion-smallness}

Now, I define a notion of smallness analogous to the one presented in \cite{Bernshteyn2024}.
By definition, the set of edges of $G$ is the continuous image of some Polish space $E$, let us say under $\pi$.
For a finite graph $H$, consider $\Hom(H,G)$ as the set of all maps

\begin{align*}
    &\varphi\colon V(H)\to V(G)\\
    &\varphi\colon E(H)\to E
\end{align*}

that satisfy that for all $uv\in E(H)$, $\pi(\varphi(uv))=\varphi(u)\varphi(v)$.
(Formally, $\varphi$ is really two maps, however I assume that $V(H)\cap E(H)=\emptyset$ and so there is never any confusion.)
Notice that $\Hom(H,G)$ is a Polish space, being a closed subset of the product space $X^{V(H)}\times(E)^{E(H)}$.

Given $\mathcal H\subseteq\Hom(H,G)$, I define $\mathcal H(u):=\{\varphi(u):\varphi\in\mathcal H\}$, which is analytic whenever $\mathcal H$ is Borel.
Also, let $\mathcal H(uv):=\{\varphi(uv):\varphi\in\mathcal H\}$.

\begin{definition}\label{def:small}~ 
    \begin{enumerate}
        \item Let $k<\omega$. I say an analytic set $A\subseteq V(G)$ has property $\Phi(A,k)$ if all odd walks of $G$ with endpoints in $A$ have length at most $2k-1$.

        \item A Borel set $\mathcal L\subseteq\Hom(L^c_n,G)$ is called \emph{tiny} if there is a vertex $u\in V(L^c_n)$ and a natural $k$ such that $\Phi(\mathcal L(u),k)$.

        \item A set $\mathcal L\subseteq\Hom(L^c_n,G)$ is \emph{small} if it is the countable union of tiny sets. This notion defines a $\sigma$-ideal.

        \item A set is \emph{large} if it is not small. Notice that when $H$ is the graph of a single vertex $H=\bullet$, I can identify $\Hom(H,G)$ with $V(G)$ and then for any $\mathcal H\subseteq\Hom(\bullet,G)$, $\mathcal H(\bullet)$ is identified with $\mathcal H$.
    \end{enumerate}
\end{definition}

I now invoke the following auxiliary result that links the $\Phi$ property to Borel 2-colourability.
The proof proceeds by induction on $n$ using analytic separation; see \cite[Claims~3.3--3.5]{L0paper} for details.

\begin{lemma}[Claim~3.5 of \cite{L0paper}]\label{lemma:L0-claim3.5}
    If $\Phi(A,n)$, then there is an invariant (meaning union of connected components) Borel set $B\supseteq[A]_{E_G}$ that induces a 2-Borel-colourable subgraph of $G$, that is, there is a Borel proper colouring $c_B\colon G\restriction B\to2$.
\end{lemma}

\begin{theorem}\label{thm:small}
    If $\chi_B(G)>2$ then $V(G)$ is large.
\end{theorem}

\begin{proof}[Proof of theorem]
    Suppose that $V(G)=\bigcup_{m<\omega}\mathcal H_m$ where each $\mathcal H_m$ is a tiny Borel set.
    By Lemma~\ref{lemma:L0-claim3.5}, there is a $G$-invariant Borel set $B_m\supseteq[\mathcal H_m]_{E_G}$ with a Borel 2-colouring $c_m$.
    Define $c\colon V(G)\to2$ by letting $m(x):=\min\{m<\omega:x\in B_m\}$ and $c(x)=c_{m(x)}(x)$.
    Then $c$ is a Borel 2-colouring because each $B_m$ is $G$-invariant.
\end{proof}

$L^c_{n+1}$ is obtained from two copies of $L^c_n$ joined by a path, so I need a way to refer to the distinct copies of $L^c_n$ inside of $L^c_{n+1}$.
By (iii) above, for every $\varphi\in\Hom(L^c_{n+1},G)$ and $i<2$, I define $\varphi^i\in\Hom(L^c_n,G)$ by $\varphi^i(u)=\varphi(u^\frown (i))$ for all $u\in V(L^c_n)$ and $\varphi^i(uv)=\varphi(u^\frown(i),v^\frown(i))$ for all $(u,v)\in E(L^c_n)$.
Thus, given $\mathcal L\subseteq\Hom(L^c_n,G)$, I define $\mathcal L^{+c}$ as the set of all $\varphi\in\Hom(L^c_{n+1},G)$ for which $\varphi^i\in\mathcal L$ for both $i<2$.
Notice that if $\mathcal L$ is Borel, then so is $\mathcal L^{+c}$.

The length of the path I add needs to be updated dynamically throughout the proof.
Formally, one adjusts the parameter $c$ by moving into a different branch of the Baire space when constructing the graphs $L^c_n$.
Notice that, for two reals $c,d\in\omega^\omega$, if $d\restriction n=c\restriction n$, then $L^c_k=L^d_k$ for all $k<n$.

\begin{lemma}\label{lemma:Ln_basics}~ 
    \begin{enumerate}[label=(\alph*)]
        \item Suppose that $n>0$, that $c$ only takes odd values, that $k$ is a natural and $t\in2^{<n}$ is such that $(p_k)^\smallfrown t^\smallfrown(i)$, for $i<2$, are vertices of $L^c_n$ (see (iii) above). Then, they are an odd distance apart in $L^c_n$.

        \item For any $n<\omega$, if $\mathcal L\subseteq\Hom(L^c_n,G)$ is a large Borel set, then for any $N<\omega$ there exists $d\in\omega^\omega$ such that $d\restriction n=c\restriction n$, $d(n)\geq N$ is odd, and $\mathcal L^{+d}$ is a nonempty subset of $\Hom(L^{d}_{n+1},G)$.

        \item For any $n<\omega$, if $\mathcal L\subseteq\Hom(L^c_n,G)$ is a large Borel set, then there exists $d\in\omega^\omega$ such that $d\restriction n=c\restriction n$ and $\mathcal L^{+d}$ is a large subset of $\Hom(L^{d}_{n+1},G)$.
    \end{enumerate}
\end{lemma}

\begin{proof}
    For (a), observe that both vertices come from the same vertex $(p_k)^\smallfrown t$ in $L^c_{|t|+1}$, and so their distance in $L^c_n$ is twice the distance from $(p_k)^\smallfrown t$ to $e_1^{|t|}$ in $L^c_{|t|+1}$ plus the length of the joining path at that stage, which is $c(|t|)+2$.
    Since $c(|t|)$ is odd, $c(|t|)+2$ is odd, and hence their total distance is even plus odd, which is odd.

    For (b), since $\mathcal L$ is large, it is not tiny.
    Hence, for every vertex $u\in V(L^c_n)$ and every natural $k$, $\Phi(\mathcal L(u),k)$ fails.
    In particular, fix the gluing vertex $u_0:=e_1^n\in V(L^c_n)$ (recall that $e_1^0=p_0$).
    Then for every $k$ there exist $\varphi_0,\varphi_1\in\mathcal L$ and an odd walk $(v_0,v_1,\dots,v_m)$ in $G$ with $v_0=\varphi_0(u_0)$, $v_m=\varphi_1(u_0)$, and $m>2k-1$.
    Choosing $k$ large enough, I can ensure that $m\geq N+2$ and $m$ is odd.
    Since $m$ is odd, $d(n):=m-2$ is also odd and satisfies $d(n)\geq N$.
    For each $1\leq j\leq m$, since $(v_{j-1},v_j)\in E(G)=\mathrm{im}(\pi)$, I pick $e_j\in E$ with $\pi(e_j)=(v_{j-1},v_j)$.
    Setting $d(x)=c(x)$ for all $x\neq n$ and noting $L^c_n=L^d_n$ (since $d\restriction n=c\restriction n$), I define $\varphi\in\Hom(L_{n+1}^d,G)$ by $\varphi^0=\varphi_0$, $\varphi^1=\varphi_1$, $\varphi(p_i)=v_{i+1}$ for $0\leq i\leq m-2=d(n)$, and the edge assignments $\varphi(u_0^\frown (0),p_0)=e_1$, $\varphi(p_j,p_{j+1})=e_{j+1}$ for $0\leq j<m-2$, and $\varphi(p_{m-2},u_0^\frown (1))=e_m$.
    Then $\varphi\in\mathcal L^{+d}$.

    I prove (c) by contradiction.
    Recall that constructing $L^c_{n+1}$ from $L^c_n$ only requires knowing the value of $c(n)$.
    Suppose that for all $d$ that agree with $c$ except possibly in the $n$-th value, $\mathcal L^{+d}=\bigcup_m\mathcal F^d_m$ where each $\mathcal F^d_m$ is a tiny Borel subset of $\Hom(L^d_{n+1},G)$.
    For each $m$ and $d$, let $w^d_m\in V(L^d_{n+1})$ and $k^d_m<\omega$ be such that $\Phi(\mathcal F^d_m(w^d_m),k^d_m)$.
    
    Since $w^d_m\in V(L^d_{n+1})$, it is either a \emph{non-path vertex} of the form $u^\smallfrown(i)$ for some $u\in V(L^c_n)$ and $i<2$, or a \emph{path vertex} $p_j$ for some $0\leq j\leq d(n)$ in the joining path.
    I claim that in either case, there exists a vertex $u^d_m\in V(L^c_n)$ and $\ell^d_m<\omega$ such that $\Phi(\mathcal F^d_m({u^d_m}^\smallfrown(i^d_m)),\ell^d_m)$ for some $i^d_m<2$ where $u^d_m{}^\smallfrown(i^d_m)$ is a non-path vertex of $L^d_{n+1}$.

    Indeed, if $w^d_m=u^\smallfrown(i)$ is already non-path, take $u^d_m=u$, $i^d_m=i$, $\ell^d_m=k^d_m$.
    If $w^d_m=p_j$ is a path vertex of $L^d_{n+1}$, then the joining path connects $e_1^n{}^\smallfrown(0)$ to $e_1^n{}^\smallfrown(1)$ via $p_0,\dots,p_{d(n)}$.
    Since $p_j$ and $e_1^n{}^\smallfrown(0)$ are at distance $j+1$ in $L^d_{n+1}$, every homomorphism $\varphi\in\mathcal F^d_m$ maps them to vertices that are at walk-distance $j+1$ in $G$.
    Given any odd walk $(a_0,\dots,a_r)$ in $G$ with $a_0,a_r\in\mathcal F^d_m(e_1^n{}^\smallfrown(0))$, witnessed by $\varphi,\varphi'\in\mathcal F^d_m$, the walk $(\varphi(p_j),\dots,\varphi(e_1^n{}^\smallfrown(0))=a_0,\dots,a_r=\varphi'(e_1^n{}^\smallfrown(0)),\dots,\varphi'(p_j))$ has length $r+2(j+1)$, which has the same parity as $r$ (hence is also odd), and has endpoints in $\mathcal F^d_m(p_j)$.
    By $\Phi(\mathcal F^d_m(p_j),k^d_m)$, $r+2(j+1)\leq 2k^d_m-1$, so $r\leq 2(k^d_m-j-1)-1$.
    It follows that $\Phi(\mathcal F^d_m(e_1^n{}^\smallfrown(0)),\max\{0,k^d_m-j-1\})$.
    Thus, set $u^d_m=e_1^n$, $i^d_m=0$, $\ell^d_m=\max\{0,k^d_m-j-1\}$.
    
    \begin{claim}
        For each $m$ and $d$, there is a Borel set $B^d_m\supseteq\mathcal F^d_m({u^d_m}^\smallfrown(i^d_m))$ such that $\Phi(B^d_m,\ell^d_m)$.
    \end{claim}

    \begin{claimproof}
        I use the First Reflection Theorem (Lemma~\ref{lemma:first-reflection}).
        Fix $\ell=\ell^d_m$.
        It suffices to verify that $\Phi_\ell:=\{A\subseteq X:\Phi(A,\ell)\}$ is $\Pi^1_1$ on $\Sigma^1_1$.
        Let $Y$ be Polish and $A\subseteq Y\times X$ be $\Sigma^1_1$.
        Fix a Polish space $N$ and a continuous surjection $g\colon N\to A$.
        For each odd $r\geq 2\ell+1$, define
        \begin{align*}
            C_r:=\{(y,x_0,\dots,x_r,n_0,n_1,\hat e_1,\dots,\hat e_r)\in Y\times X^{r+1}\times N^2\times E^r:{}\\
            g(n_0)=(y,x_0)\land g(n_1)=(y,x_r)\land\textstyle\bigwedge_{j<r}\pi(\hat e_{j+1})=(x_j,x_{j+1})\}.
        \end{align*}
        Then $C_r$ is closed (by continuity of $g$ and $\pi$), and $D_r:=\mathrm{proj}_Y(C_r)$ is $\Sigma^1_1$.
        Hence $D:=\bigcup_{r\text{ odd},\,r\geq 2\ell+1}D_r$ is $\Sigma^1_1$, and
        $A_{\Phi_\ell}=\{y\in Y:A_y\in\Phi_\ell\}=Y\setminus D$
        is $\Pi^1_1$.
        Since $\mathcal F^d_m({u^d_m}^\smallfrown(i^d_m))$ is analytic and satisfies $\Phi(\cdot,\ell)$, the First Reflection Theorem yields the desired Borel superset $B^d_m$.
    \end{claimproof}

    Take $$\mathcal H^d_m=\{\varphi\in\mathcal L:\varphi(u^d_m)\in B^d_m\},$$ which is a Borel subset of $\Hom(L^c_n,G)$.
    I claim that $\mathcal H^d_m$ is tiny.
    Indeed, $\mathcal H^d_m(u^d_m)=\{\varphi(u^d_m):\varphi\in\mathcal L\land\varphi(u^d_m)\in B^d_m\}\subseteq B^d_m$, so $\Phi(\mathcal H^d_m(u^d_m),\ell^d_m)$.
    
    Now, the set of pairs $(u^d_m,i^d_m)$ ranges over the finite set $V(L^c_n)\times2$ (which does not depend on $d$), and $(m,d)$ range over countably many values.
    Thus $\bigcup_{m,d}\mathcal H^d_m$ is a countable union of tiny sets, hence small.
    The set $\mathcal L_-:=\mathcal L\setminus\bigcup_{m,d}\mathcal H^d_m$ is therefore large (and Borel).
    By part (b), there exist $\varphi$ and $d$ such that $\varphi\in\mathcal L_-^{+d}\subseteq\mathcal L^{+d}$.
    Then $\varphi\in\mathcal F^d_m$ for some $m$.
    It follows that $\varphi^{i^d_m}(u^d_m)=\varphi({u^d_m}^\smallfrown(i^d_m))\in\mathcal F^d_m({u^d_m}^\smallfrown(i^d_m))\subseteq B^d_m$, hence $\varphi^{i^d_m}\in\mathcal H^d_m$.
    But $\varphi\in\mathcal L_-^{+d}$ implies $\varphi^{i^d_m}\in\mathcal L_-$, contradicting $\mathcal L_-\cap\mathcal H^d_m=\emptyset$.
\end{proof}

\section{Construction of a minimal graph of Borel chromatic number at least 3}\label{sec:L0-construction}

The goal of this section is to construct a family of graphs $\mathbb L_c$ indexed by reals $c\in\omega^\omega$.

Now, fix $c$ and consider $X_c$ as the set of all tuples $(m,k,x)\in\mathbb N\times\mathbb N\times2^\mathbb N$ such that either $m=0$ and $k=0$, or $m\geq1$ and $k\leq c(m-1)$; this is a closed subspace of the product space, and hence Polish.
Define, for each $(m,k,x)$ with $m\leq n$, $\pi_n(m,k,x)$ as the vertex of $L^c_n$ determined by these parameters.
Formally, $\pi_n(m,k,x):=(p_k)^\smallfrown x\restriction(n-m)$.
For example, $\pi_3(1,2,0110\cdots)=(p_2,0)$ can be seen in Figure~\ref{fig:L_n} labeled as $(2,0)$.

Finally, $\mathbb L_c$ is the graph on $X_c$ where two $(n_i,k_i,x_i)$, for $i<2$, are adjacent if and only if the $\pi_n(n_i,k_i,x_i)$ are adjacent in all $L^c_n$ with $n\geq\max\{n_0,n_1\}$.
Some basic properties of $\mathbb L_c$ follow.

\begin{lemma}\label{lemma:L0_basics}~ 
    \begin{enumerate}[label=(\alph*)]
        \item Two vertices $(n_i,k_i,x_i)\in X_c$, for $i<2$, are in the same connected component of $\mathbb L_c$ if and only if there are $t_i\in2^{<\omega}$ and $x\in2^\omega$ such that $|t_0|-|t_1|=n_1-n_0$ and $x_i=t_i^\smallfrown x$.
        
        \item $\mathbb L_c$ has no cycles and all vertices of $\mathbb L_c$ have degree 2 except for the vertex $(0,0,\{0\}^\omega)$, which has degree 1.
        
        \item If, additionally, $c$ takes only odd values, then $\chi_B(\mathbb L_c)=3$.
    \end{enumerate}
\end{lemma}

\begin{proof}
    For (a), let $(n_i,k_i,x_i)\in X_c$ for $i<2$.
    If they are in the same connected component, then for all large enough $n$, $\pi_n(n_0,k_0,x_0)$ and $\pi_n(n_1,k_1,x_1)$ lie in the same copy of $L^c_{n-1}$ inside $L^c_n$.
    By the construction of $L^c_n$, two vertices $(p_{k_0})^\smallfrown x_0\restriction(n-n_0)$ and $(p_{k_1})^\smallfrown x_1\restriction(n-n_1)$ are in the same copy of $L^c_{n-1}$ in $L^c_n$ if and only if their last binary digits agree, that is, $x_0(n-n_0-1)=x_1(n-n_1-1)$ for all large $n$.
    Writing $x_0=t_0^\smallfrown x$ and $x_1=t_1^\smallfrown x$ with $|t_0|=n_1-n_0+|t_1|$ (assuming $n_0\leq n_1$), we obtain the stated condition.
    Conversely, if $x_i=t_i^\smallfrown x$ with $|t_0|-|t_1|=n_1-n_0$, then for all large $n$, $\pi_n(n_0,k_0,x_0)=(p_{k_0})^\smallfrown t_0^\smallfrown x\restriction r$ and $\pi_n(n_1,k_1,x_1)=(p_{k_1})^\smallfrown t_1^\smallfrown x\restriction r$ for some $r$, so they share a common tail and lie in the same connected component of $L^c_n$, hence of $\mathbb L_c$.

    For (b), every vertex of $L^c_n$ has degree at most $2$, so the same holds for $\mathbb L_c$.
    It follows from acyclicity of the finite paths $L^c_n$ that $\mathbb L_c$ is acyclic.
    The vertex $(0,0,0^\omega)$ maps under $\pi_n$ to $e_0^n=(p_0)^\smallfrown 0^n$, which is an endpoint of $L^c_n$, hence has degree $1$ in $L^c_n$ for all $n\geq1$.
    Thus $(0,0,0^\omega)$ has degree $1$ in $\mathbb L_c$.
    For any other vertex $(m,k,x)\in X_c$, the image $\pi_n(m,k,x)$ is an interior vertex of $L^c_n$ for all sufficiently large $n$, so it has degree exactly $2$.

    (c) That $\chi_B(\mathbb L_c)\leq3$ follows from Proposition~\ref{prop:brooks} and part (b).
    I argue the other inequality by contradiction.
    Suppose that there is a Borel partition of $X_c$ into two independent sets $X_c=I_0\cup I_1$.
    Since $X_c$ is Polish, the Baire category theorem guarantees that some $I_\varepsilon$ is non-meager, hence comeager in a basic open set $U=\{(m,k,x)\in X_c:x\restriction r=t\}$ for some $m$, $k$, $r$, and $t\in2^r$.
    In particular, for a comeager set of $x\in2^\mathbb N$ extending $t$, $(m,k,x)\in I_\varepsilon$.
    Since $(m,k,t^\smallfrown(0)^\smallfrown y)$ and $(m,k,t^\smallfrown(1)^\smallfrown y)$ are in the same connected component (by part (a)) for any $y\in2^\mathbb N$, and $I_\varepsilon$ is comeager in $U$, I can find $y$ such that both $(m,k,t^\smallfrown(0)^\smallfrown y)$ and $(m,k,t^\smallfrown(1)^\smallfrown y)$ lie in $I_\varepsilon$.
    But by Lemma~\ref{lemma:Ln_basics}(a), these two vertices are an odd distance apart in $\mathbb L_c$ (as $(p_k)^\smallfrown t^\smallfrown(0)$ and $(p_k)^\smallfrown t^\smallfrown(1)$ are an odd distance apart in every $L^c_n$ for large enough $n$), contradicting $I_\varepsilon$ being independent.
\end{proof}

The main result is then split into two parts.

\begin{theorem}\label{thm:L0-hom-equivalent-main}
    For any analytic graph $G$, either $\chi_B(G)\leq 2$ or there is a $c$ such that $\mathbb L_c\to_cG$.
\end{theorem}

Moreover, $c$ can be chosen to be unbounded and take only odd values.

\begin{theorem}\label{thm:L0-hom-equivalent}
    Let $c,d\in\omega^\omega$ be unbounded and only taking odd values, then $\mathbb L_c$ and $\mathbb L_d$ are continuously homomorphically equivalent.
\end{theorem}

The proof of Theorem~\ref{thm:L0-hom-equivalent} is purely combinatorial: one constructs homomorphisms between $\mathbb L_c$ and $\mathbb L_d$ by iteratively extending partial maps along finite path components, using the ``large gap property'' guaranteed by unboundedness.
The argument does not depend on the method used to establish Theorem~\ref{thm:L0-hom-equivalent-main}; see \cite[Proposition~4.2 and Claim~4.1]{L0paper} for details.

Combining Theorems~\ref{thm:L0-hom-equivalent-main} and \ref{thm:L0-hom-equivalent}, and fixing any unbounded odd $c_0$ (e.g., $c_0(0)=1$ and $c_0(n)=2n-1$ for $n>0$), we recover the $L_0$ dichotomy of \cite{L0paper}:

\begin{corollary}
    Let $\mathbb L_0:=\mathbb L_{c_0}$.
    For any analytic graph $G$ on a Polish space, exactly one of the following holds:
    \begin{enumerate}
        \item $\chi_B(G)\leq 2$;
        \item $\mathbb L_0\to_cG$.
    \end{enumerate}
\end{corollary}

\section{Proof of Theorem~\ref{thm:L0-hom-equivalent-main}}

By Lemma~\ref{lemma:L0_basics}(c) and Fact~\ref{fact:chrom_transitive}, it is clear that both conditions cannot hold simultaneously.
Now suppose that $G$ is an analytic graph of Borel chromatic number at least three.
The theorem is proved by constructing a valid $c$ and a continuous homomorphism witnessing $\mathbb L_c\to_cG$.

By Theorem~\ref{thm:small}, $V(G)$ is large.
Fix compatible complete metrics $\rho$ on $V(G)$ and $\delta$ on $E$.
Since $L^c_0=\bullet$, I identify $\Hom(L^c_0,G)$ with $V(G)$ and set $\mathcal L_0:=V(G)$, which is large.
I recursively construct a function $c\in\omega^\omega$ (taking only odd values) and a sequence of large Borel sets $\mathcal L_n\subseteq\Hom(L^c_n, G)$ such that for all $n$,

\begin{enumerate}
    \item $\mathcal L_{n+1}\subseteq\mathcal L_n^{+c}$;
    \item for all $u\in V(L^c_n)$, $\mathrm{diam}_\rho(\overline{\mathcal L_n(u)})< 2^{-n}$; and
    \item for all $(u,v)\in E(L^c_n)$, $\mathrm{diam}_\delta(\overline{\mathcal L_n(u,v)})< 2^{-n}$.
\end{enumerate}

Given $\mathcal L_n$ large, apply Lemma~\ref{lemma:Ln_basics}(c) to obtain $d\in\omega^\omega$ with $d\restriction n=c\restriction n$ and $\mathcal L_n^{+d}$ large; set $c(n):=d(n)$.
To arrange conditions (2) and (3) for $n+1$: since $V(G)$ and $E$ can each be covered by countably many closed sets of $\rho$-diameter (respectively $\delta$-diameter) less than $2^{-(n+1)}$, and small sets form a $\sigma$-ideal, one may intersect with the preimages of these covers and discard the resulting small pieces, passing to a large Borel subset $\mathcal L_{n+1}\subseteq\mathcal L_n^{+c}$ satisfying (2) and (3) at stage $n+1$.
Moreover, by choosing $c(n)$ sufficiently large using part (b) at each stage, $c$ can be made unbounded.

By (1), for all $n\geq m$, $$\mathcal L_{n+1}(\pi_{n+1}(m,k,x))\subseteq\mathcal L_n(\pi_n(m,k,x));$$ indeed, each $\varphi\in\mathcal L_{n+1}$ satisfies $\varphi^{x(n-m)}\in\mathcal L_n$ and $$\varphi^{x(n-m)}(\pi_n(m,k,x))=\varphi(\pi_{n+1}(m,k,x)).$$
By the same reasoning, if $\pi_n(m_0,k_0,x_0)$ and $\pi_n(m_1,k_1,x_1)$ are adjacent in $L^c_n$ for all large $n$, then \begin{align*}
    \mathcal L_{n+1}(\pi_{n+1}(m_0,k_0,x_0),\pi_{n+1}(m_1,k_1,x_1))\\
    \subseteq\mathcal L_n(\pi_n(m_0,k_0,x_0),\pi_n(m_1,k_1,x_1)).
\end{align*}
For $(m,k,x)\in X_c$, let $f(m,k,x)$ be the unique point in $$\bigcap_{n=m}^\infty\overline{\mathcal L_n(\pi_n(m,k,x))}=\bigcap_{n=0}^\infty\overline{\mathcal L_{n+m}((p_k)^\frown x\restriction n)}.$$
The map $f\colon X_c\to X$ is continuous.

It remains to check that it is a homomorphism from $\mathbb L_c$ to $G$.
Suppose that $(m_i,k_i,x_i)_{i<2}$ are adjacent in $\mathbb L_c$, that is, for all $n\geq n_0:=\max\{m_0,m_1\}$, $\pi_n(m_i,k_i,x_i)$ are adjacent in $L_n^c$.
There is a unique edge $e\in E$ inside $$\bigcap_{n\geq n_0}\overline{\mathcal L_n(\pi_n(m_0,k_0,x_0),\pi_n(m_1,k_1,x_1))}.$$
By continuity of $\pi$, $(f(m_0,k_0,x_0),f(m_1,k_1,x_1))=\pi(e)\in G$.

\nocite{*}
\bibliographystyle{alphaurl}
\bibliography{bibliography/references}

\end{document}